\documentclass[10pt,a4paper]{article}
\usepackage{amsmath,amssymb,amsthm,esint,bm}
\usepackage{mathrsfs}
\usepackage{bookmark}
\usepackage{enumerate}
\allowdisplaybreaks[3]

\newtheorem{definition}{Definition}[section]
\newtheorem{theorem}[definition]{Theorem}
\newtheorem{lemma}[definition]{Lemma}

\newtheorem{corollary}[definition]{Corollary}
\theoremstyle{remark}
\newtheorem{remark}[definition]{Remark}
\numberwithin{equation}{section}

\title{Global weighted gradient estimates for nonlinear p-Laplacian type elliptic equations and its application}

\author{Xuehui Hao \\
School of Mathematical Sciences \\
Nankai University, Tianjin, 300071, China\\
e-mail:\,2120170064@mail.nankai.edu.cn} 
\date{\today}

\usepackage{hyperref}
\begin{document}

\maketitle
\begin{abstract}
We obtain the global weighted $W^{1,p}$ estimates for weak solutions of nonlinear elliptic equations over Reifenberg flat domains. Where nonlinearity $A(x,z,\xi)$ is assumed to be local uniform continuous in $z$ and have small BMO semi-norm in $x$. Moreover, we derive Besov regularity for solutions of a class of special harmonic equations by making use of $W^{1,p}$ estimate. \\

\end{abstract}


\section{Introduction and main results.}\label{section1}

\subsection{Introduction.}
In this paper we consider the following nonlinear elliptic equations:

\begin{equation}\label{model}
  \left\{\begin{array}{r@{\  }c@{\  }ll}
  \operatorname{div}A(x,u,\nabla u)&=&\operatorname{div}\left(|F|^{p-2}F\right)  & \mbox{in}\ \ \Omega\,, \\[0.05cm]
 u&=&0  & \mbox{on}\ \ \partial \Omega. \\[0.05cm]
\end{array}\right.
\end{equation}
where $p\in(1,\infty)$, $\Omega\subset\mathbb{R}^{n}$, $n\geq2$ is a bounded and generally irregular domain. $F$ is a given measurable vector field function.
The solution $u:\Omega\longrightarrow\mathbb{R}$ is a real-valued unknown function. The nonlinearity  $A=A(x,z,\xi):\Omega\times R\times\mathbb{R}^{n}\rightarrow\mathbb{R}^{n}$ is differentiable with respect to $\xi\neq0$. Moreover, $A(x,z,\xi)$ is assumed to have local uniform continuity in $z$, i.e.
\begin{equation}\label{c-condition}
  |A(x,z_{1},\xi)-A(x,z_{2},\xi)|\leq\omega_{M}(|z_{1}-z_{2}|)|\xi|^{p-1}
\end{equation}
for almost every $x\in\Omega$, all $z_{1},z_{2}\in[-M,M]$. Where $\omega_{M}:\mathbb{R}^{+}\rightarrow\mathbb{R}^{+}$ is modulus of continuity with  $\lim\limits_{\rho\rightarrow0^{+}}\omega_{M}(\rho)=0$, monotonically non-decreasing and concave. And we further assume that there exists a constant $\Lambda>0$ such that
\begin{equation}\label{growth}
  \left\{
   \begin{array}{c}
  |A(x,z,\xi)|+|\partial_{\xi}A(x,z,\xi)||\xi| \leq\Lambda|\xi|^{p-1}   \\
  \left< \partial_{\xi}A(x,z,\xi)\zeta,\zeta \right> \geq \Lambda^{-1}|\xi|^{p-2}|\zeta|^{2}.  \\
   \end{array}
  \right.
  \end{equation}
for almost every $x\in \Omega$, all $z\in\mathbb{R}$ and all $\xi,\zeta \in \mathbb{R}^{n}\backslash\{0\}$. Furthermore, we require some more regularity on nonlinearity, namely we assume $A(x,z,\xi)$ is measurable in $\Omega$ for every $(z,\xi)\in\mathbb{R}\times\mathbb{R}^{n}\setminus\{0\}$ and has a sufficiently small $BMO$ (bounded mean oscillation) semi-norm in $x$. More precise description of these structural requirements will be given in the next subsection. As usual, we consider a function $u\in W_{0}^{1,p}(\Omega)$, which is a weak solution of \eqref{model} with $F\in L^{p}(\Omega, \mathbb{R}^{n})$, if
\begin{equation*}
  \int_{\Omega}\left<A(x,u,\nabla u),\nabla \varphi\right>\operatorname{d}\!x=\int_{\Omega}\left<|F|^{p-2}F,\nabla \varphi\right>\operatorname{d}\!x
\end{equation*}
for any test function $\varphi\in W_{0}^{1,p}(\Omega)$.

As a classical topic in the regularity theory of solutions to partial differential equations and systems, Calder\'on-Zygmund theory has been the theme of a number of contributions with different peculiarities. This theory traces its origins back to works of Calder\'on and Zygmund \cite{Calderon1952On} in 1950s. They proved the $L^{p}$-estimate for the gradient of solutions to linear elliptic equations in the whole $\mathbb{R}^{n}$ by establishing the standard Calder\'on-Zygmund theory of singular integrals. As for the case of parabolic equations, that's Fabes's contribution \cite{Fabes1966Singular}. For the nonlinear Calder\'on-Zymund theory, Iwaniec \cite{iwaniec1983projections} first derived the Calder\'on-Zymund estimates for the $p$-Laplace equations via the sharp maximal operators and priori regularity estimates. As for weighted case, Mengesha and Phuc obtained the global regularity estimates in weighted Lorentz spaces, see \cite{mengesha2012global}.Caffarelli and Peral \cite{Caffarelli1998On} obtained the $W^{1,p}$ regularity of solutions to fully nonlinear elliptic equations. In the case when $A=A(x,\nabla u)$, the results has been obtained by many researchers, see \cite{byun2012nonlinear} for classical Lebesgue spaces and \cite{Byun2013Global} for weighted Lebesgue spaces. As for the case $A(x,u,\nabla u)$, the authors succeeded to obtain interior gradient estimates when $u$ is bounded, see \cite{nguyen2016interior}. In the recent paper \cite{byun2018global}, the authors obtained global gradient estimates of \eqref{model} for classical Lebesgue spaces in the case when $u\in L^{\infty}(\Omega)$.

As for Besov regularity, see \cite{clop2019besov}\cite{ma2019higher}, in which the case that $A$ is independent on $z$ and corresponding obstacle problems have been studied. In the process, Calder\'on-Zygmund estimate play a crucial role.

The present article is a natural outgrowth of \cite{byun2018global} and deals with global weighted $W^{1,p}$ theory for \eqref{model}. In particular, we derive an extended version of the global $W^{1,p}$ estimate in the settings of the weighted Lorentz space. At the end of the paper, we derive Besov regularity for solutions of a class of special harmonic equations by making use of Calder\'on-Zygmund estimate.

This paper is organized as follows. In the next subsection, we give some notations and precise statement of the main results. In Section\ref{section2}, we state some elementary estimates which will be used frequently in the paper. In Section\ref{section3} we present weighted good-$\lambda$ type inequality that will be essential for the proof of the main theorem. In Section\ref{section4}, the desired global weighted estimate is obtain. The last section contains the proof of Besov regularity for solutions.

\subsection{Notations and main results.}
Let us start by introducing a few notations to be used in what follows.

Throughout the paper, we denote by $|U|$ the integral $\int_{U}\operatorname{d}\!x$ for every measurable set $U\subset\mathbb{R}^{n}$. For an open set $\Omega\subset\mathbb{R}^{n}$, $\Omega_{\rho}(x)\triangleq\Omega\cap B_{\rho}(x)$, where $B_{\rho}(x)$ is a n-dimensional open ball. For the sake of convenience and simplicity, we employ the letter $C>0$ to denote any constants which can be explicitly computed in terms of known quantities such as $n,p,q$. Thus the exact value denoted by $C$ may change from line to line in a given computation.

To measure the oscillation of $A(x,z,\xi)$ in $x$-variables on $B_{\rho}(y)$, we consider a function defined by
\begin{equation}\label{def of theta}
  \theta \left(A,B_{\rho}(y)\right)(x,z)=\sup_{\xi\in\mathbb{R}^{n}\setminus\{0\}}\frac{|A(x,z,\xi)-\bar{A}_{B_{\rho}(y)}(z,\xi)|}{|\xi|^{p-1}}
\end{equation}
where
\begin{equation*}
  \bar{A}_{B_{\rho}(y)}(z,\xi)=\fint_{B_{\rho}(y)}A(x,z,\xi)\operatorname{d}\!x
\end{equation*}

In order to state our main results, we introduce the following definitions.

\begin{definition}
The domain is said to be $(\delta,R)$-Reifenberg flat if there exist postive constants $\delta$ and $R$ with the property that for each $x_{0}\in\partial\Omega$ and each $\rho\in(0,R)$, there exist a local coordinate system $\{x_{1},\cdots,x_{n}\}$ with origin at the point $x_{0}$ such that
\begin{equation*}
  B_{\rho}(x_{0})\cap\{x:x_{n}>\rho\delta\}\subset B_{\rho}(x_{0})\cap\Omega\subset B_{\rho}(x_{0})\cap\{x:x_{n}>-\rho\delta\}
\end{equation*}

\end{definition}
\begin{definition}
Let $1<q<\infty$, a non-negative, locally integrable function $\omega:\mathbb{R}\rightarrow [0, \infty)$ is said to be in the class $A_{q}$ of Muckenhoupt weight if
\begin{equation*}
  [\omega]_{q}:=\sup_{balls B\subset \mathbb{R}^{n}}\left(\fint_{B}\omega(x)\operatorname{d}\!x\right)\left(\fint_{B}\omega(x)^{\frac{1}{1-q}}\operatorname{d}\!x\right)^{q-1}<+\infty.
\end{equation*}
\end{definition}

\begin{definition}
The weighted Lorentz space $L_{\omega}^{q,t}(\Omega)$ with $0<q<\infty$, $0<t\leq\infty$, is the set of measurable functions $g$ on $\Omega$ such that
\begin{equation*}
  \|g\|_{L_{\omega}^{q,t}(\Omega)}:=\left(q\int_{0}^\infty\left(\alpha^{q}\omega(\{x\in\Omega:|g(x)|>\alpha\})\right)^{\frac{t}{q}}
  \frac{\operatorname{d}\!\alpha}{\alpha}\right)^{\frac{1}{t}}<+\infty
\end{equation*}
when $t\neq\infty$; for $t=\infty$ the space $L_{\omega}^{q,\infty}(\Omega)$ is set to be the usual Marcinkiewica space with quasinorm
\begin{equation*}
  \|g\|_{L_{\omega}^{q,\infty}(\Omega)}:=\sup_{\alpha>0}\alpha\omega(\{x\in\Omega:|g(x)|>\alpha\})^{\frac{1}{q}}.
\end{equation*}
\end{definition}

\begin{remark}
When $t=q$, the Lorentz space $L_{\omega}^{q,q}(\Omega)$ is equivalent to weighted Lebesgue space $L_{\omega}^{q}(\Omega)$, whose norm is defined by
$$\|g\|_{L_{\omega}^{q}(\Omega)}:=\left(\int_{\Omega}|g(x)|^{q}\omega(x)\operatorname{d}\!x\right)^{\frac{1}{q}}$$
\end{remark}

The main result of this paper is the following global regularity estimates for weak solutions of \eqref{model} in weighted Lorentz space.
\begin{theorem}\label{1-1}
Let $p,q,\gamma\geq1$. Then, there exists a sufficiently small constant $\delta=\delta(p,q,n,\Lambda,\gamma,M,\omega_{M})>0$ such that the following statement holds true. For a given vector field $F\in L_{\omega}^{pq,t}(\Omega,\mathbb{R}^{n})$, $0<t\leq\infty$, if $u\in W_{0}^{1,p}(\Omega)\cap L^{\infty}(\Omega)$ satisfying $\|u\|_{L^{\infty}(\Omega)}\leq M$  is a weak solution of \eqref{model} with $A(x,z,\xi)$ satisfying \eqref{c-condition}, \eqref{growth} and
\begin{equation}\label{small BMO}
  \sup_{-M\leq z\leq M}\sup_{0<\rho\leq R}\sup_{y\in \mathbb{R}^{n}}\fint_{B_{\rho}(y)}\theta\left(A,B_{\rho}(y)\right)(x,z)\operatorname{d}\!x\leq\delta
\end{equation}
for some $R>0$. $\Omega$ is $(\delta,R)$-Reifenberg flat. Then the following weighted regularity estimate holds.
\begin{equation*}
  \|\nabla u\|_{L_{\omega}^{pq,t}(\Omega)}\leq C\|F\|_{L_{\omega}^{pq,t}(\Omega)}
\end{equation*}
where $\omega\in A_{q}$ with $[\omega]_{q}\leq\gamma$, $\theta\left(A,B_{\rho}(y)\right)$ is defined in \eqref{def of theta} and $C$ is a constant depending on $n$, $p$, $q$, $\Lambda$, $\gamma$, $M$, $\omega_{M}$, $\Omega$.
\end{theorem}

As for the interior case, the proof is similar to that of  global case. Thus, we only state the result.
\begin{theorem}\label{1-1-1}
Let $p,q,\gamma\geq1$. Then, there exists a sufficiently small constant $\delta=\delta(p,q,n,\Lambda,\gamma,M,\omega_{M})>0$ such that the following statement holds true. For a given vector field $F\in L_{\omega}^{pq,t}(B_{2R},\mathbb{R}^{n})$, $0<t\leq\infty$, if $u\in W_{loc}^{1,p}(B_{2R})\cap L^{\infty}(B_{2R})$ satisfying $\|u\|_{L^{\infty}(B_{2R})}\leq M$  is a weak solution of

\begin{equation*}
  \operatorname{div}A(x,u,\nabla u)=\operatorname{div}\left(|F|^{p-2}F\right)  \quad\quad \mbox{in}\ \ B_{2R}\,
\end{equation*}
with $A(x,z,\xi)$ satisfying \eqref{c-condition}, \eqref{growth} and
\begin{equation}\label{small BMO}
  \sup_{-M\leq z\leq M}\sup_{0<\rho\leq R}\sup_{y\in B_{R}}\fint_{B_{\rho}(y)}\theta\left(A,B_{\rho}(y)\right)(x,z)\operatorname{d}\!x\leq\delta
\end{equation}
for some $R>0$.  Then the following weighted regularity estimate holds.
\begin{equation*}
  \|\nabla u\|_{L_{\omega}^{pq,t}(B_{R})}\leq C\left(\|F\|_{L_{\omega}^{pq,t}(B_{2R})}+\omega(B_{2R})^{1/pq}\left(\fint_{B_{2R}}|\nabla u|^{p}\operatorname{d}\!x\right)^{1/p}\right)
\end{equation*}
where $\omega\in A_{q}$ with $[\omega]_{q}\leq\gamma$, $\theta\left(A,B_{\rho}(y)\right)$ is defined in \eqref{def of theta} and $C$ is a constant depending on $n$, $p$, $q$, $\Lambda$, $\gamma$, $M$, $\omega_{M}$, $R$.
\end{theorem}

In order to state the other main result, which is actually a consequence of Theorem\ref{1-1-1}, we recall the Besov space $B_{p,q}^{\alpha}(\mathbb{R}^{n})$.

\begin{definition}
Let $h\in\mathbb{R}^{n}$, $f:\mathbb{R}^{n}\rightarrow\mathbb{R}$. Let $0<\alpha<1$ and $1\leq p,q<\infty$. The Besov space consists of all functions $f\in L^{p}(\mathbb{R}^{n})$ for which the norm
\begin{equation*}
  \|f\|_{B_{p,q}^{\alpha}(\mathbb{R}^{n})}=\|f\|_{L^{p}(\mathbb{R}^{n})}+[f]_{\dot{B}_{p,q}^{\alpha}(\mathbb{R}^{n})}
\end{equation*}
is finite. Where
\begin{equation*}
  [f]_{\dot{B}_{p,q}^{\alpha}(\mathbb{R}^{n})}=\left(\int_{\mathbb{R}^{n}}\left(\int_{\mathbb{R}^{n}}\frac{|f(x+h)-f(x)|^{p}}{|h|^{\alpha p}}\operatorname{d}\!x\right)^{\frac{q}{p}}\frac{\operatorname{d}\!h}{|h|^{n}}\right)^{\frac{1}{q}}.
\end{equation*}
When $q=\infty$, we say that $f\in B_{p,\infty}^{\alpha}$, if
\begin{equation*}
  \|f\|_{B_{p,\infty}^{\alpha}(\mathbb{R}^{n})}=\|f\|_{L^{p}(\mathbb{R}^{n})}+[f]_{\dot{B}_{p,\infty}^{\alpha}(\mathbb{R}^{n})}
\end{equation*}
is finite. Where
\begin{equation*}
  [f]_{\dot{B}_{p,\infty}^{\alpha}(\mathbb{R}^{n})}=\sup_{h\in\mathbb{R}^{n}}\left(\int_{\mathbb{R}^{n}}\frac{|f(x+h)-f(x)|^{p}}{|h|^{\alpha p}}\operatorname{d}\!x\right)^{\frac{1}{p}}.
\end{equation*}
\end{definition}

\begin{remark}
As matter of fact, one can simply integrates for $h\in B_{\delta}$ for a fixed $\delta>0$ when $q<\infty$ and take the supremum over $|h|\leq\delta$ to obtain an equivalent norm.
\end{remark}

\begin{theorem}\label{1-2}
Let $0<\alpha<1$, Assume that $A(x,z,\xi)$ satisfies \eqref{c-condition} and \eqref{growth} for $p=2$, take $\omega_{M}(t)=t^{\alpha}$. Moreover, we suppose that there exists $g\in L_{loc}^{\frac{n}{\alpha}}(\Omega)$ such that
\begin{equation}\label{3.4}
  |A(x,z,\xi)-A(y,z,\xi)|\leq|x-y|^{\alpha}(g(x)+g(y))|\xi|
\end{equation}
for a.e.$x\in\Omega$, $\forall(z,\xi)\in\mathbb{R}\times\mathbb{R}^{n}$. If $u\in W_{loc}^{1,2}(\Omega)\cap L^{\infty}(\Omega)$ is a weak solution of
\begin{equation}\label{specialmodel}
   \operatorname{div}A(x,u,\nabla u)=0   \quad\quad\mbox{in}\ \ \Omega\,,
\end{equation}
then, $\nabla u\in B_{2,\infty}^{\alpha}$, locally.
\end{theorem}

\section{Preliminaries.}\label{section2}

\subsection{Invariance.}\label{section2.1}
We note that our equation is scaling invariant. Indeed, if $A(x,u,\nabla u)$ satisfies the conditions \eqref{c-condition}, \eqref{growth} and \eqref{small BMO}, then for some fixed $\mu,r >0$, $x_{0}\in\mathbb{R}$, the rescaled nonlinearity
\begin{equation*}
  \hat{A}(x,z,\xi)=\frac{A(rx+x_{0},\mu rz, \mu\xi)}{\mu^{p-1}}
\end{equation*}
satisfies \eqref{growth}. Moreover, $\hat{A}(x,z,\xi)$ satisfies
\begin{equation}\label{small BMO2}
  \sup_{-\frac{M}{\mu r}\leq z\leq \frac{M}{\mu r}}\sup_{0<\rho\leq \frac{R}{r}}\sup_{y\in \mathbb{R}^{n}}\fint_{B_{\rho}(y)}\theta\left(A,B_{\rho}(y)\right)(x)\operatorname{d}\!x\leq\delta
\end{equation}
and
\begin{equation}
  |\hat{A}(x,z_{1},\xi)-\hat{A}(x,z_{2},\xi)|\leq\omega_{M}(\mu r|z_{1}-z_{2}|)|\xi|^{p-1}
\end{equation}
for a.e. $x\in\widehat{\Omega}$, $\forall z_{1},z_{2}\in\left[-\frac{M}{\mu r},\frac{M}{\mu r}\right]$. Where $\widehat{\Omega}=\left\{\frac{x-x_{0}}{r}, x\in\Omega\right\}$ is $\left(\delta,\frac{R}{r}\right)$-Reifenberg flat.

The properties mentioned above are obvious owing to some elementary calculation. Let us now consider the invariance of equation \eqref{model} with respect to scaling. Assume that $u\in W_{0}^{1,p}(\Omega)\cap L^{\infty}(\Omega)$ is a weak solution of \eqref{model}, then $\hat{u}=u(rx+x_{0})/\mu\in W_{0}^{1,p}(\widehat{\Omega})\cap L^{\infty}(\widehat{\Omega})$ satisfying $\|\hat{u}\|_{L^{\infty}(\widehat{\Omega})}\leq\frac{M}{\mu r}$ solve the equation
\begin{equation}\label{model11}
  \left\{\begin{array}{r@{\  }c@{\  }ll}
  \operatorname{div}\hat{A}(x,\hat{u},\nabla\hat{u})&=&\operatorname{div}\left(|\hat{F}|^{p-2}\hat{F}\right)  & \mbox{in}\ \ \Omega\,, \\[0.05cm]
 \hat{u}&=&0  & \mbox{on}\ \ \partial \Omega. \\[0.05cm]
\end{array}\right.
\end{equation}
where $\hat{F}(x)=\frac{F(rx+x_{0})}{\mu}$.

\subsection{Muckenhoupt weights and weighted inequalities.}

We will use the strong doubling property of $A_{q}$ weight stated below. Hereafter we denote by $\omega(\Omega)$ the integral $\int_{\Omega}\omega(x)\operatorname{d}\!x$

\begin{lemma}\label{strong doubling}{\rm (cf.\cite{coifman1974weighted})}.
For  $1<q<\infty$, the following statements hold true

\begin{enumerate}[(1)]
\item if $\omega\in A_{q}$, then for every ball $B\subset\mathbb{R}^{n}$ and every measurable set $E\subset B$,

\begin{equation*}
  \omega(B)\leq [\omega]_{q}\left(\frac{|B|}{|E|}\right)^{q}\omega(E)
\end{equation*}
\item if $\omega\in A_{q}$ with $[\omega]_{q}\leq \gamma$ for some given $\gamma \geq 1$, then there is $C=C(\gamma,n)$ and $\alpha=\alpha(\gamma, n)>0$ such that
    \begin{equation*}
      \omega(E)\leq C\left(\frac{|E|}{|B|}\right)^{\alpha}\omega(B)
    \end{equation*}
    for every ball $B\subset\mathbb{R}^{n}$ and every measurable set $E\subset B$.
\end{enumerate}
\end{lemma}

\begin{lemma}\label{jie}{\rm (cf.\cite{grafakos2008classical})}.
Let $\omega$ be an $A_{q}$ weight for some $1<q<\infty$. Then there exists $\sigma=\sigma(n,q,[\omega]_{q})>0$ such that $q-\sigma>1$ and $\omega\in A_{q-\sigma}$ with $[\omega]_{q-\sigma}\leq C(n,q,[\omega]_{q})$.
\end{lemma}

secondly, we state the following result which comes from standard measure theory.
\begin{lemma}\label{2-1}
Assume that $g\geq 0$ is a measurable function in a bounded subset $U\subset \mathbb{R}^{n}$. Let $\theta>0$, $\Gamma>1$ be constants, and let $\omega$ be a weight in $\mathbb{R}^{n}$. Then for $0<q,t<\infty$, we have
\begin{equation*}
  g\in L_{\omega}^{q,t}(U)\Leftrightarrow S:= \sum_{k\geq1}\Gamma^{tk}\omega \left(\{x\in U : g(x)>\theta \Gamma^{k}\}\right)^{\frac{t}{q}}<+\infty
\end{equation*}
and moreover, there exist a constant $C>0$ depending only on $\theta, \Gamma, t$, such that
\begin{equation*}
  C^{-1}S\leq \|g\|_{L_{\omega}^{q,t}(U)}^{t}\leq C\left(\omega(U)^{\frac{t}{q}}+S\right)
\end{equation*}
Analogously, for $0<q<\infty$ and $t=\infty$ we have
\begin{equation*}
   C^{-1}T\leq \|g\|_{L_{\omega}^{q,\infty}(U)}\leq C\left(\omega(U)^{\frac{1}{q}}+T\right)
\end{equation*}
Where $T$ is the quantity
\begin{equation*}
  T:=\sup_{k\geq1}\Gamma^{k}\omega \left(\{x\in U : g(x)>\theta \Gamma^{k}\}\right)^{\frac{1}{q}}
\end{equation*}
\end{lemma}

The following is a summary of embedding theorems that will be used later, see \cite{grafakos2008classical}.
\begin{lemma}\label{embedding}
Let $\Omega$ be a bounded measurable subset of $\mathbb{R}^{n}$ and $\omega$ be an $A_{q}$ weight for $1<q<\infty$.
\begin{enumerate}[(1)]
\item If $0<t\leq p_{1}<p_{2}\leq\infty$, then $L_{\omega}^{p_{2},\infty}(\Omega)\subset L_{\omega}^{p_{1},t}(\Omega)$. Moreover
      $$\|g\|_{L_{\omega}^{p_{1},t}(\Omega)}\leq C(p_{1},p_{2},t)\omega(\Omega)^{\frac{1}{p_{1}}-\frac{1}{p_{2}}}\|g\|_{L_{\omega}^{p_{2},\infty}(\Omega)}$$
\item If $0<t\leq\infty$, $0<q<\infty$, then $L_{\omega}^{q,t}(\Omega)\subset L_{\omega}^{q,\infty}(\Omega)$.
\end{enumerate}
\end{lemma}

Thirdly, we concern on the connection between the boundedness of the Hardy-Littlewood maximal operator on weighted spaces and the characterization of $A_{q}$ weight, which is crucial in treating our problem. For a given locally integrable function $f\in L_{\operatorname{loc}}^{1}(\mathbb{R}^{n})$, the Hardy-Littlewood maximal function is defined as
\begin{equation*}
  \mathcal{M}f(x)=\sup_{\rho>0}\fint_{B_{\rho}(x)}|f(y)|\operatorname{d}\!y
\end{equation*}
For a function $f$ that is defined only on a bounded domain $U$, we define
\begin{equation*}
  \mathcal{M}_{U}f(x)=\mathcal{M}(f\chi U)(x),
\end{equation*}
Where $\chi U$ is the characteristic function of the set $U$. The following boundedness of Hardy-Littlewood maximal operator $\mathcal{M}: L_{\omega}^{q,t}(\mathbb{R}^{n})\rightarrow L_{\omega}^{q,t}(\mathbb{R}^{n})$ is classical.

\begin{lemma}\label{2-2}{\rm (cf.\cite{mengesha2012global}\cite{muckenhoupt1972weighted})}.
Let $\omega$ be an $A_{q}$ weight for some $1<q<\infty$. For any $0<t\leq\infty$, there exists a constant $C=C(n, q, t, [\omega]_{q})$ such that
\begin{equation}\label{boundedness of M}
  \|\mathcal{M}f\|_{L_{\omega}^{q,t}(\mathbb{R}^{n})}\leq C \|f\|_{L_{\omega}^{q,t}(\mathbb{R}^{n})}
\end{equation}
for all $f\in L_{\omega}^{q,t}(\mathbb{R}^{n})$. Conversely, if \eqref{boundedness of M} holds for all $f\in L_{\omega}^{q,t}(\mathbb{R}^{n})$, then $\omega$ must be an $A_{q}$ weight.
\end{lemma}

Finally, we recall the following technical lemma, which will be used in the proof of the weighted estimates, which is originally due to \cite{krylov1979estimate}\cite{safonov1983harnack}. The version given below is proved in \cite{mengesha2011weighted}
\begin{lemma}\label{2-3}
Let $\Omega$ be a $(\delta,R)$-Reifenberg flat domain with $\delta<\frac{1}{4}$, Suppose $\omega\in A_{q}$ with $[\omega]_{q}\leq\gamma$ for some $1<q<\infty$ and some $\gamma\geq 1$. Suppose also that  $C,D$ are measurable sets satisfying $C\subset D\subset \Omega$ and there are $\rho_{0} \in \left(0,\frac{R}{2000}\right)$ such that the sequence of balls $\{B_{\rho_{0}}(y_{i})\}_{i=1}^{L}$ with centers $y_{i}\in\overline{\Omega}$ covers $\Omega$, Assume that $\epsilon \in (0,1)$ such that the followings hold,
\begin{enumerate}[(1)]
\item $\omega(C)<\epsilon\,\omega\left(B_{\rho_{0}}(y_{i})\right)$ for all $i=1,\cdots L$,
\item for all $x\in \Omega$ and $\rho \in(0,2\rho_{0})$, if $\omega(C\cap B_{\rho}(x))\geq \epsilon \,\omega(B_{\rho}(x))$, then $B_{\rho}(x)\cap \Omega\subset D$.
\end{enumerate}
Then
\begin{equation*}
  \omega(C)\leq\epsilon_{1}\omega(D), \,\,\,\,\,\,for\,\,\,\,\,\, \epsilon_{1}=\epsilon\left(\frac{10}{1-4\delta}\right)^{nq}\gamma^{2}.
\end{equation*}
\end{lemma}

\subsection{A known approximation estimate.}

For the sake of convenience and simplicity, we use the notation $u,F,A$ and $\Omega$ instead of $\hat{u},\hat{F},\hat{A}$ and $\widehat{\Omega}$ respectively.
Let $\sigma\geq6$ be a universal constant, let $u$ be a weak solution of

\begin{equation*}
  \left\{\begin{array}{r@{\  }c@{\  }ll}
  \operatorname{div}A(x,u,\nabla u)&=& \operatorname{div}(|F|^{p-2}F)  & \mbox{in}\ \ \Omega_{\sigma}\,, \\[0.05cm]
  u&=&0  & \mbox{on}\ \ \partial \Omega_{\sigma}. \\[0.05cm]
\end{array}\right.\eqno(2.5)
\end{equation*}

We consider the limiting problem
\begin{itemize}
  \item interior case:
        $$\operatorname{div}\bar{A}(\nabla h)=0\quad \quad \operatorname{in}\,\,\,\, B_{4}\eqno(2.6)$$
  \item boundary case£º
  \begin{equation*}
  \left\{\begin{array}{r@{\  }c@{\  }ll}
  \operatorname{div}\bar{A}(\nabla h)&=& 0  & \mbox{in}\ \ B_{4}^{+}\,, \\[0.05cm]
  h&=&0  & \mbox{on}\ \  B_{4}\cap\{x_{n}=0\}, \\[0.05cm]
\end{array}\right.\eqno(2.7)
\end{equation*}

\end{itemize}
for the interior case, $\bar{A}(\xi)$ is given by
\begin{equation*}
  \bar{A}(\xi)=\fint_{B_{4}}A(x,\bar{u}_{\Omega_{5}},\xi)\operatorname{d}\!x
\end{equation*}
for the boundary case, $\bar{A}(\xi)$ is given by
\begin{equation*}
  \bar{A}(\xi)=\frac{1}{|B_{4}|}\int_{B_{4}^{+}}A(x,\bar{u}_{\Omega_{5}},\xi)\operatorname{d}\!x
\end{equation*}
where
\begin{equation*}
  \bar{u}_{\Omega_{5}}=\fint_{\Omega_{5}}u(x)\operatorname{d}\!x.
\end{equation*}

We recall a known approximation estimate established in \cite{byun2018global}. This approximation estimate will be used in the proof of Theorem\ref{1-1}.
\begin{lemma}\label{comparison1}(interior case)
For some fixed $\epsilon\in(0,1)$, there exists a constants $\sigma=\sigma(n,p,\Lambda,\omega_{M},M,\epsilon)\geq6$ such that $u\in W_{0}^{1,p}(B_{\sigma})$ is a weak solution of $(2.5)$ with $ \|u\|_{L^{\infty}(B_{\sigma})}\leq\frac{M}{\mu r}$ and satisfies

\begin{equation*}
  \frac{1}{|B_{\sigma}|}\int_{B_{\sigma}}|\nabla u|^{p}\operatorname{d}\!x\leq1
\end{equation*}
Suppose also that there exists some positive number $\delta=\delta(\Lambda,\omega_{M},n,p,M,\epsilon)\in(0,\frac{1}{8})$ such that
\begin{equation*}
  \frac{1}{|B_{5}|}\int_{B_{5}}\theta(A,B_{5})(x,\bar{u}_{B_{5}})\operatorname{d}\!x\leq\delta
\end{equation*}
and
\begin{equation*}
   \frac{1}{|B_{\sigma}|}\int_{B_{\sigma}}|F|^{p}\operatorname{d}\!x\leq\delta^{p}
\end{equation*}
Then there exists a weak solution $h\in W^{1,p}(B_{4})$of $(2.6)$ such that the following inequality holds
\begin{equation*}
 \|\nabla h\|_{L^{\infty}(B_{3})}\leq C\quad and \quad\frac{1}{|B_{4}|}\int_{B_{4}}|\nabla u-\nabla h|^{p}\operatorname{d}\!x\leq\epsilon^{p}.
\end{equation*}
Where $C=C(n,p,\Lambda)>1$.
\end{lemma}

\begin{lemma}\label{comparison2}(boundary case)
For some fixed $\epsilon\in(0,1)$, there exists a constants $\sigma=\sigma(\Lambda,\omega_{M},n,p,M,\epsilon)\geq6$ such that $u\in W_{0}^{1,p}(\Omega_{\sigma})$ is a weak solution of $(2.5)$ with $ \|u\|_{L^{\infty}(\Omega_{\sigma})}\leq\frac{M}{\mu r}$ and satisfies

\begin{equation*}
  \frac{1}{|B_{\sigma}|}\int_{\Omega_{\sigma}}|\nabla u|^{p}\operatorname{d}\!x\leq1.
\end{equation*}
Suppose also that there exists some positive number $\delta=\delta(\Lambda,\omega_{M},n,p,M,\epsilon)\in(0,\frac{1}{8})$ such that
\begin{equation*}
  B_{5}^{+}\subset\Omega_{5}\subset B_{5}\cap\{x:x_{n}>-10\delta\},
\end{equation*}
\begin{equation*}
  \frac{1}{|B_{5}|}\int_{\Omega_{5}}\theta(A,\Omega_{5})(x,\bar{u}_{\Omega_{5}})\operatorname{d}\!x\leq\delta,
\end{equation*}
and
\begin{equation*}
   \frac{1}{|B_{\sigma}|}\int_{\Omega_{\sigma}}|F|^{p}\operatorname{d}\!x\leq\delta^{p}.
\end{equation*}
Then there exists a weak solution $h\in W^{1,p}(B_{4}^{+})$of $(2.7)$ such that the following inequality holds
\begin{equation*}
 \|\nabla \bar{h}\|_{L^{\infty}(\Omega_{3})}\leq C\quad and \quad\frac{1}{|B_{4}|}\int_{\Omega_{4}}|\nabla u-\nabla \bar{h}|^{p}\operatorname{d}\!x\leq\epsilon^{p}
\end{equation*}
Where $\bar{h}$ is the zero extension of $h$ from $B_{4}^{+}$ to $B_{4}$, $C=C(\Lambda,n,p)>1$.
\end{lemma}

\section{Weighted estimates.}\label{section3}

\begin{lemma}\label{3-2}
Let $p\geq1$, $\gamma>1$ and $\epsilon>0$ sufficiently small. Then there exists sufficiently large number $N=N(n,p,\Lambda)>1$, some positive number $\delta=\delta(n,p,\Lambda,\epsilon,\gamma, M,\omega_{M})>0$ and $\sigma=\sigma(n,p,\Lambda,\epsilon, M,\omega_{M})\geq6$ such that the following statement holds. Suppose that $u\in W_{0}^{1,p}(\Omega)$ is a weak solution of (1.1) with $\|u\|_{L^{\infty}(\Omega)}\leq M$ and the nonlinearity $A(x,z,\xi)$ satisfies \eqref{small BMO}. If $\Omega$ is a $(\delta,R)$-Reifenberg flat domain and for $\forall y\in\Omega$, $\forall r\in\left(0,\frac{R}{\sigma}\right]$, we have
\begin{equation*}
   B_{r}(y)\cap\left\{x\in\Omega:\mathcal{M}(|\nabla u|^{p})\leq\left(\frac{6}{7}\right)^n\mu^{p}\right\}\cap\left\{x\in\Omega:\mathcal{M}(|F|^{p})\leq \left(\frac{6}{7}\right)^n\mu^{p}\delta^{p}\right\}\neq \emptyset
\end{equation*}
then
\begin{equation*}
   \omega\left(B_{r}(y)\cap\left\{x\in\Omega:\mathcal{M}(|\nabla u|^{p})>\left(\frac{6}{7}\right)^n\mu^{p} N^p\right\}\right)<\epsilon\omega(B_{r}(y))
\end{equation*}
for $\omega\in A_{q}$ with $[\omega]_{q}\leq\gamma$ and $q>1$.
\end{lemma}
\begin{proof}
We divide the proof into two steps.

$Step1$. We begin by proof an unweighted estimate.

Suppose that $\hat{u}\in W_{0}^{1,p}(\widehat{\Omega})$ is a weak solution of (2.5) with $\|\hat{u}\|_{L^{\infty}(\widehat{\Omega})}\leq\frac{M}{\mu r}$ and the nonlinearity $\hat{A}(x,z,\xi)$ satisfies
\begin{equation}\label{small BMO1}
  \sup_{-\frac{M}{\mu r}\leq z\leq \frac{M}{\mu r}}\sup_{0<\rho\leq \sigma}\sup_{y\in \mathbb{R}^{n}}\fint_{B_{\rho}(y)}\theta\left(A,B_{\rho}(y)\right)(x,z)\operatorname{d}\!x\leq\delta.
\end{equation}
If $\widehat{\Omega}$ is a $(\delta,\sigma)$-Reifenberg flat domain and
\begin{equation}\label{3.7}
   B_{1}\cap\left\{x\in\widehat{\Omega}:\mathcal{M}(|\nabla \hat{u}|^{p})\leq\left(\frac{6}{7}\right)^n\right\}\cap\left\{x\in\widehat{\Omega}:\mathcal{M}(|\hat{F}|^{p})\leq \left(\frac{6}{7}\right)^n\delta^{p}\right\}\neq \emptyset
\end{equation}
then, we claim that
\begin{equation}\label{3-1}
   \left|B_{1}\cap\left\{x\in\widehat{\Omega}:\mathcal{M}(|\nabla\hat{u}|^{p})>\left(\frac{6}{7}\right)^n N^p\right\}\right|<\epsilon\left|B_{1}\right|
\end{equation}

In fact, For a given $\epsilon>0$, let $\epsilon'=\epsilon'(n,p,\Lambda,\epsilon)>0$ be a positive number to be determined later. Then, let $\delta=\delta(n,p,\Lambda,\epsilon',M,\omega_{M})>0$ be the number defined in Lemma\ref{comparison1} and Lemma\ref{comparison2}. We prove the claim \eqref{3-1} with this choice of $\delta$. By the assumption \eqref{3.7}, we can discover that there exists $x_{0}$ such that
\begin{equation}\label{4.3}
  x_{0}\in B_{1}\cap\left\{x\in\widehat{\Omega}:\mathcal{M}(|\nabla \hat{u}|^{p})\leq\left(\frac{6}{7}\right)^{n}\right\}\cap\left\{x\in\widehat{\Omega}:\mathcal{M}(|\hat{F}|^{p})\leq \left(\frac{6}{7}\right)^{n}\delta^{p}\right\}
\end{equation}
Since $x_{0}\in B_{1}$, we can easily obtain $B_{\rho}\subset B_{\rho+1}(x_{0})$. For $\forall\rho\geq6$, it follows that
\begin{equation*}
  \frac{1}{|B_{\rho}|}\int_{\widehat{\Omega}_{\rho}}|\nabla \hat{u}|^{p}\operatorname{d}\!x\leq\left(\frac{\rho+1}{\rho}\right)^{n}\frac{1}{|B_{\rho+1}(x_{0})|}\int_{\widehat{\Omega}_{\rho+1}(x_{0})}|\nabla \hat{u}|^{p}\operatorname{d}\!x\leq\left(\frac{7}{6}\right)^{n}\left(\frac{6}{7}\right)^{n}=1
\end{equation*}
\begin{equation*}
  \frac{1}{|B_{\rho}|}\int_{\widehat{\Omega}_{\rho}}| \hat{F}|^{p}\operatorname{d}\!x\leq\left(\frac{7}{6}\right)^{n}\frac{1}{|B_{\rho+1}(x_{0})|}\int_{\widehat{\Omega}_{\rho+1}(x_{0})}| \hat{F}|^{p}\operatorname{d}\!x\leq\delta^{p}.
\end{equation*}
Owing to the nonlinearity $\hat{A}(x,z,\xi)$ satisfies \eqref{small BMO1}, all conditions in Lemma\ref{comparison1} and Lemma\ref{comparison2} are satisfied. Thus, one can find $H\in L^{\infty}(\widehat{\Omega}_{3})$ such that
\begin{equation}\label{4.4}
 \frac{1}{|B_{4}|}\int_{\widehat{\Omega}_{4}}|\nabla\hat{u}-H|^{p}\operatorname{d}\!x\leq C(n)\epsilon'^{p},\ \ \ \|H\|_{L^{\infty}(\widehat{\Omega}_{3})}\leq C_{*}
\end{equation}
Take $N^{p}=\max\{4^{p}\left(\frac{7}{6}\right)^{n}C_{*}^{p},2^{n}\}$, we claim that
\begin{equation}\label{4.5}
 B_{1}\cap\{x\in \widehat{\Omega}:\mathcal{M}_{\widehat{\Omega}_{4}}\left(|\nabla \hat{u}-H|^{p}\right)(x)\leq C_{*}^{p}\}\subset
  B_{1}\cap\left\{x\in \widehat{\Omega}:\mathcal{M}(|\nabla \hat{u}|^{p})(x)\leq \left(\frac{6}{7}\right)^{n}N^{p}\right\}
\end{equation}
In order to prove this statement, assume that $x$ is a point in the set on the left side of \eqref{4.5}, for any $r'>0$, if $r'<2$, note that $B_{r'}(x)\subset B_{3}$, as a result, we have
\begin{eqnarray*}
   && \left(\frac{1}{|B_{r'}(x)|}\int_{\widehat{\Omega}_{r'}(x)}|\nabla \hat{u}(z)|^{p}\operatorname{d}\!z\right) ^{\frac{1}{p}}\\
   &\leq& 2\left(\frac{1}{|B_{r'}(x)|}\int_{\widehat{\Omega}_{r'}(x)}|\nabla \hat{u}(z)-H(z)|^{p}\operatorname{d}\!z\right) ^{\frac{1}{p}}+2\left(\frac{1}{|B_{r'}(x)|}\int_{\widehat{\Omega}_{r'}(x)}|H|^{p}\operatorname{d}\!z\right) ^{\frac{1}{p}} \\
   &\leq& 2\left(\mathcal{M}_{\widehat{\Omega}_{4}}\left(|\nabla u-H|^{p}\right)(x)\right)^{\frac{1}{p}}+2\|H\|_{L^{\infty}(\widehat{\Omega}_{3})} \\
   &\leq& 4C_{*} \\
   &\leq& \left(\frac{6}{7}\right)^{\frac{n}{p}}N
\end{eqnarray*}
If $r'\geq2$, then $B_{r'}(x)\subset B_{2r'}(x_{0})$, we have from this and \eqref{4.3} that
\begin{eqnarray*}
  \frac{1}{|B_{r'}(x)|}\int_{\widehat{\Omega}_{r'}(x)}|\nabla \hat{u}(z)|^{p}\operatorname{d}\!z
   &\leq& \left(\frac{2r'}{r'}\right)^{n}\frac{1}{|B_{2r'}(x_{0})|}\int_{\widehat{\Omega}_{2r'}(x_{0})}|\nabla \hat{u}(z)|^{p}\operatorname{d}\!z \\
   &\leq& 2^{n}\mathcal{M}(|\nabla \hat{u}|^{p})(x_{0}) \\
   &\leq& 2^{n}\left(\frac{6}{7}\right)^{n} \\
   &\leq& \left(\frac{6}{7}\right)^{n}N^{p}
\end{eqnarray*}
Hence, we have proved that \eqref{4.5} holds. It follows that
\begin{equation*}
B_{1}\cap\left\{x\in \widehat{\Omega}:\mathcal{M}(|\nabla \hat{u}|^{p})(x)> \left(\frac{6}{7}\right)^{n}N^{p}\right\}\subset E:=B_{1}\cap\left\{x\in \widehat{\Omega}:\mathcal{M}_{\widehat{\Omega}_{4}}\left(|\nabla \hat{u}-H|^{p}\right)(x)> C_{*}^{p}\right\}
\end{equation*}
In addition, owing to the weak (1,1)-type estimate of Hardy-Littlewood maximal function, we have
\begin{equation*}
  |E|\leq\frac{C(n)}{C_{*}^{p}}\int_{\widehat{\Omega}_{4}}|\nabla \hat{u}-H|^{p}\operatorname{d}\!z
\end{equation*}
Then we can get
\begin{equation}\label{4.6}
  \frac{|E|}{|B_{1}|}\leq\frac{C(n)}{C_{*}^{p}}\frac{1}{|B_{4}|}\int_{\widehat{\Omega}_{4}}|\nabla \hat{u}-H|^{p}\operatorname{d}\!z\leq C'(n,p,\Lambda)\epsilon'^{p}
\end{equation}
where the last inequality is due to \eqref{4.4}. Finally, the estimate of \eqref{3-1} follows by making use of the definition of $E$ and choosing $\epsilon'=\epsilon'(n,p,\Lambda,\epsilon)$ such that $ C'(n,p,\Lambda,\gamma)\epsilon'^{p}=\epsilon$

$Step2.$ We will use properties of $A_{q}$ weights and the translation scaling invariance of Lebesgue measure to obtain a weighted version.

For $\forall y\in\Omega$, define
\begin{equation*}
  \widehat{\Omega}=\left\{\frac{x-y}{r},x\in\Omega\right\}\quad\quad\hat{A}(x,z,\xi)=\frac{A(rx+y,\mu rz,\mu\xi)}{\mu^{p-1}}
\end{equation*}
\begin{equation*}
  \hat{u}(x)=\frac{u(rx+y)}{\mu r}\quad\quad\quad\hat{F}(x)=\frac{F(rx+y)}{\mu}
\end{equation*}
then, $\hat{A}(x,z,\xi)$ satisfies \eqref{small BMO1}, $\hat{u}\in W_{0}^{1,p}(\widehat{\Omega})$ is weak solution of (2.5) with $\|\hat{u}\|_{L^{\infty}(\widehat{\Omega})}\leq\frac{M}{\mu r}$ and $\widehat{\Omega}$ is $(\delta,\frac{R}{r})$-Reifenberg flat domain. By the assumption, there exists $x_{0}\in\Omega_{\rho}(y)$ such that

\begin{equation*}
  \sup_{\rho}\frac{1}{|B_{\rho}(x_{0})|}\int_{\Omega_{\rho}(x_{0})}|\nabla u|^{p}\operatorname{d}\!x\leq\left(\frac{6}{7}\right)^{n}\mu^{p}
\end{equation*}
and
\begin{equation*}
  \sup_{\rho}\frac{1}{|B_{\rho}(x_{0})|}\int_{\Omega_{\rho}(x_{0})}|F|^{p}\operatorname{d}\!x\leq\left(\frac{6}{7}\right)^{n}\mu^{p}\delta^{p}
\end{equation*}
then we can derive that $z_{0}=\frac{x_{0}-y}{r}\in B_{1}$ and $z_{0}\in\widehat{\Omega}$, it follows that
\begin{eqnarray*}
  \mathcal{M}(|\nabla \hat{u}|^{p})(z_{0})
  &=& \sup_{\rho}\frac{1}{|B_{\rho}(z_{0})|}\int_{\widehat{\Omega}_{\rho}(z_{0})}|\nabla\hat{u}(z)|^{p}\operatorname{d}\!z \\
  &=& \sup_{\rho}\frac{1}{|B_{\rho}(z_{0})|}\int_{\widehat{\Omega}_{\rho}\left(\frac{x_{0}-y}{r}\right)}|\nabla u(rz+y)|^{p}\mu^{-p}\operatorname{d}\!z \\
  &=& \mu^{-p}\sup_{\rho}\frac{1}{|B_{\rho}(z_{0})|}\int_{\Omega_{r\rho}(x_{0})}|\nabla u(t)|^{p}r^{-n}\operatorname{d}\!t \\
  &=& \mu^{-p}\sup_{\rho}\frac{1}{|B_{r\rho}(x_{0})|}\int_{\Omega_{r\rho}(x_{0})}|\nabla u(t)|^{p}\operatorname{d}\!t\\
  &=& \mu^{-p}\mathcal{M}(|\nabla u|^{p})(x_{0})\\
  &\leq&\left(\frac{6}{7}\right)^{n}
\end{eqnarray*}
Similarily,
\begin{equation*}
  \mathcal{M}(|\hat{F}|^{p})(z_{0})=\mu^{-p}\mathcal{M}(|F|^{p})(x_{0})\leq\left(\frac{6}{7}\right)^{n}\delta^{p}.
\end{equation*}
Hence, all conditions in $Step1$ are satisfied and as can be seen from the above process
\begin{equation}\label{trans1}
  \mathcal{M}(|\nabla \hat{u}|^{p})\left(\frac{x-y}{r}\right)=\mu^{-p} \mathcal{M}(|\nabla u|^{p})(x)\quad and\quad
  \mathcal{M}(|\hat{F}|^{p})\left(\frac{x-y}{r}\right)=\mu^{-p} \mathcal{M}(|F|^{p})(x)
\end{equation}
From $Step1$, we have
\begin{equation*}
  \left|B_{1}\cap\left\{z\in\widehat{\Omega}:\mathcal{M}(|\nabla\hat{u}|^{p})(z)>\left(\frac{6}{7}\right)^n N^p\right\}\right|<\epsilon\left|B_{1}\right|
\end{equation*}
Since Lebesgue measure is scale and translation invariant, it follows that
\begin{equation*}
  \left|B_{r}(y)\cap\left\{x\in\Omega:\mathcal{M}(|\nabla u|^{p})(x)>\left(\frac{6}{7}\right)^n\mu^{p} N^p\right\}\right|<\epsilon\left|B_{r}(y)\right|
\end{equation*}
where we used \eqref{trans1}. Combining this and Lemma\ref{strong doubling}(2), we can derive that
\begin{equation*}
   \omega\left(B_{r}(y)\cap\left\{x\in\Omega:\mathcal{M}(|\nabla u|^{p})>\left(\frac{6}{7}\right)^n\mu^{p} N^p\right\}\right)<C\epsilon^{\alpha}\omega(B_{r}(y))
\end{equation*}
Thus, the Lemma follows in view of the arbitrariness of $\epsilon$.
\end{proof}

\begin{lemma}\label{4-2}
Let $p\geq1$, $\gamma>1$ $\sigma=\sigma(n,p,\Lambda,\epsilon, M,\omega_{M})\geq6$ and $\epsilon>0$ sufficiently small. Let $\{B_{r}(y_{i})\}_{i=1}^{L}$ be a sequence of balls with centers $y_{i}\in\overline{\Omega}$ and a common radius $0<r<\frac{R}{400\sigma}$ Then there exists sufficiently large number $N=N(n,p,\Lambda)>1$ and some positive number $\delta=\delta(n,p,\Lambda,\epsilon,\gamma,M,\omega_{M})>0$, such that the following statement holds. Suppose that $u\in W_{0}^{1,p}(\Omega)$ is a weak solution of \eqref{model} with $\|u\|_{L^{\infty}(\Omega)}\leq M$ and the nonlinearity $A(x,z,\xi)$ satisfies \eqref{small BMO}. If $\Omega$ is a $(\delta,R)$-Reifenberg flat domain and the following inequality holds
\begin{equation}\label{4.8}
  \omega\left(\left\{x\in\Omega:\mathcal{M}(|\nabla u|^{p})>\left(\frac{6}{7}\right)^{n}\mu^{p}N^{p}\right\}\right)\leq \epsilon\omega(B_{r}(y_{i}))
\end{equation}
for some $\omega\in A_{q}$, $q>1$ and $[\omega]_{q}\leq\gamma$. Then, we have
\begin{eqnarray}\label{4.9}
  &&\omega\left(\left\{x\in\Omega:\mathcal{M}(|\nabla u|^{p})>\left(\frac{6}{7}\right)^{n}\mu^{p}N^{p}\right\}\right)\nonumber\\
  &\leq&\epsilon_{1}\omega\left(\left\{x\in\Omega:\mathcal{M}(|\nabla u|^{p})>\left(\frac{6}{7}\right)^{n}\mu^{p}\right\}\right)+\epsilon_{1}\omega\left(\left\{x\in\Omega:\mathcal{M}(|F|^{p})>\left(\frac{6}{7}\right)^{n}\mu^{p}\delta^{p}\right\}\right)
\end{eqnarray}
where $\epsilon_{1}$ is defined in Lemma\ref{2-3}
\end{lemma}

\begin{proof}
Let $N$, $\delta$ be defined as in Lemma\ref{3-2}, let
\begin{equation*}
  C=\left\{x\in\Omega:\mathcal{M}(|\nabla u|^{p})(x)>\left(\frac{6}{7}\right)^{n}\mu^{p}N^{p}\right\}
\end{equation*}
and
\begin{equation*}
  D=\left\{x\in\Omega:\mathcal{M}(|\nabla u|^{p})(x)>\left(\frac{6}{7}\right)^{n}\mu^{p}\right\}\cup\left\{x\in \Omega:\mathcal{M}(|F|^{p})(x)>\left(\frac{6}{7}\right)^{n}\mu^{p}\delta^{p}\right\}
\end{equation*}
by applying Lemma\ref{2-3} and Lemma\ref{3-2}, we can complete the proof of the Lemma.
\end{proof}

\begin{corollary}\label{4-3}
Let $p\geq1$, $\gamma>1$ and let $\Omega,\{B_{r}(y_{i})\}_{i=1}^{L},\epsilon,N,\delta$ be as in Lemma\ref{4-2}. Suppose that $u\in W_{0}^{1,p}(\Omega)$ is a weak solution of \eqref{model} with $\|u\|_{L^{\infty}(\Omega)}\leq M$ and the nonlinearity $A(x,z,\xi)$ satisfies \eqref{small BMO}. If
\begin{equation}\label{4.10}
  \omega\left(\left\{x\in\Omega:\mathcal{M}(|\nabla u|^{p})>\left(\frac{6}{7}\right)^{n}\mu^{p}N^{p}\right\}\right)\leq \epsilon\omega(B_{r}(y_{i}))
\end{equation}
for some $\omega\in A_{q}$, $q>1$ and $[\omega]_{q}\leq\gamma$. For $\forall\beta>0$, set $\epsilon_{2}=\max\{1,2^{\beta-1}\}\epsilon_{1}^{\beta}$, where $\epsilon_{1}$ is defined in Lemma\ref{2-3}, then we have
\begin{eqnarray*}\label{4.11}
   \omega\left(\left\{x\in\Omega:\mathcal{M}(|\nabla u|^{p})>\left(\frac{6}{7}\right)^{n}\mu^{p}N^{pk}\right\}\right)^{\beta}
   &\leq&\epsilon_{2}^{k}\omega\left(\left\{x\in\Omega:\mathcal{M}(|\nabla u|^{p})>\left(\frac{6}{7}\right)^{n}\mu^{p}\right\}\right)^{\beta}\nonumber\\
   &+&\sum_{i=1}^k\epsilon_{2}^{i}\omega\left(\left\{x\in\Omega:\mathcal{M}(|F|^{p})>\left(\frac{6}{7}\right)^{n}\mu^{p}\delta^{p}N^{p(k-i)}\right\}\right)^{\beta}
\end{eqnarray*}
\end{corollary}

\begin{proof}
We now prove this corollary by induction. The case $k=1$ follows from Lemma\ref{4-2}, suppose now that the conclusion is true for some $k>1$. Let $u_{N}=\frac{u}{N}$ and $f_{N}=\frac{f}{N}$, we discover that
\begin{eqnarray}
   \omega\left(\left\{x\in\Omega:\mathcal{M}(|\nabla u_{N}|^{p})>\left(\frac{6}{7}\right)^{n}\mu^{p}N^{p}\right\}\right)
   &=&  \omega\left(\left\{x\in\Omega:\mathcal{M}(|\nabla u|^{p})>\left(\frac{6}{7}\right)^{n}\mu^{p}N^{2p}\right\}\right)\nonumber\\
   &\leq& \omega\left(\left\{x\in\Omega:\mathcal{M}(|\nabla u|^{p})>\left(\frac{6}{7}\right)^{n}\mu^{p}N^{p}\right\}\right) \nonumber\\
   &\leq& \epsilon \omega(B_{r}(y_{i}))
\end{eqnarray}
for $i=1,\cdots,L$. Where the second inequality holds because of $N>1$ and the last one is due to assumption \eqref{4.10}. Now by induction assumption it follows that
\begin{eqnarray*}
 && \omega\left(\left\{x\in\Omega:\mathcal{M}(|\nabla u|^{p})>\left(\frac{6}{7}\right)^{n}\mu^{p}N^{p(k+1)}\right\}\right)^{\beta}\nonumber\\
   &=& \omega\left(\left\{x\in\Omega:\mathcal{M}(|\nabla u_{N}|^{p})>\left(\frac{6}{7}\right)^{n}\mu^{p}N^{pk}\right\}\right)^{\beta} \nonumber\\
   &\leq& \epsilon_{2}^{k}\omega\left(\left\{x\in\Omega:\mathcal{M}(|\nabla u_{N}|^{p})>\left(\frac{6}{7}\right)^{n}\mu^{p}\right\}\right)^{\beta}
   + \sum_{i=1}^k\epsilon_{2}^{i}\omega\left(\left\{x\in\Omega:\mathcal{M}(|F_{N}|^{p})>\left(\frac{6}{7}\right)^{n}\mu^{p}\delta^{p}N^{p(k-i)}\right\}\right)^{\beta} \nonumber\\
   &=& \epsilon_{2}^{k}\omega\left(\left\{x\in\Omega:\mathcal{M}(|\nabla u|^{p})>\left(\frac{6}{7}\right)^{n}\mu^{p}N^{p}\right\}\right)^{\beta}
   + \sum_{i=1}^k\epsilon_{2}^{i}\omega\left(\left\{x\in\Omega:\mathcal{M}(|F|^{p})>\left(\frac{6}{7}\right)^{n}\mu^{p}\delta^{p}N^{p(k+1-i)}\right\}\right)^{\beta} \nonumber\\
   &\leq& \epsilon_{2}^{k}\left(\epsilon_{2}\omega\left(\left\{x\in\Omega:\mathcal{M}(|\nabla u|^{p})>\left(\frac{6}{7}\right)^{n}\mu^{p}\right\}\right)^{\beta}
   +\epsilon_{2}\omega\left(\left\{x\in\Omega:\mathcal{M}(|F|^{p})>\left(\frac{6}{7}\right)^{n}\mu^{p}\delta^{p}\right\}\right)^{\beta}\right)  \nonumber\\
   &+& \sum_{i=1}^k\epsilon_{2}^{i}\omega\left(\left\{x\in\Omega:\mathcal{M}(|F|^{p})>\left(\frac{6}{7}\right)^{n}\mu^{p}\delta^{p}N^{p(k+1-i)}\right\}\right)^{\beta} \nonumber\\
   &=& \epsilon_{2}^{k+1}\omega\left(\left\{x\in\Omega:\mathcal{M}(|\nabla u|^{p})>\left(\frac{6}{7}\right)^{n}\mu^{p}\right\}\right)^{\beta}
   + \sum_{i=1}^{k+1}\epsilon_{2}^{i}\omega\left(\left\{x\in\Omega:\mathcal{M}(|F|^{p})>\left(\frac{6}{7}\right)^{n}\mu^{p}\delta^{p}N^{p{k+1-i}}\right\}\right)^{\beta}
\end{eqnarray*}
Here we have used the case $k=1$ to the first term in the forth inequality. Hence we complete the proof of the corollary.
\end{proof}

\section{Weighted Lorentz estimates.}\label{section4}

Before proving the main result, we provide some elementary estimates that will be crucial for obtaining the Calder\'on-Zygmund type estimates.
\begin{lemma}\label{5-2}{\rm (cf.\cite{nguyen2016interior}\cite{tolksdorf1984regularity})}.
Let $p>1$ and $\Omega\subset\mathbb{R}^{n}$ be a bounded open set. Assume that $A(x,z,\xi)$ satisfies \eqref{growth}. Then for any $\xi_{1},\xi_{2}\in W^{1,p}(\Omega)$ and any nonnegative function $\phi\in C(\overline{\Omega})$, it holds that
\begin{enumerate}[(1)]
  \item If $1<p<2$, then for any $\tau>0$,
  \begin{eqnarray*}
    \int_{\Omega}|\nabla\xi_{1}-\nabla\xi_{2}|^{p}\phi\operatorname{d}\!x
    &\leq& \tau\int_{\Omega}|\nabla\xi_{1}|^{p}\phi\operatorname{d}\!x \\
    &+&C(\tau,p,\Lambda)\int_{\Omega}\left<A(x,\xi_{1},\nabla\xi_{1})-A(x,\xi_{2},\nabla\xi_{2}),\nabla\xi_{1}-\nabla\xi_{2}\right>\phi\operatorname{d}\!x
  \end{eqnarray*}
  \item If $p\geq2$, then
  \begin{equation*}
    \int_{\Omega}|\nabla\xi_{1}-\nabla\xi_{2}|^{p}\phi\operatorname{d}\!x\leq C(p,\Lambda)\int_{\Omega}\left<A(x,\xi_{1},\nabla\xi_{1})-A(x,\xi_{1},\nabla\xi_{2}),\nabla\xi_{1}-\nabla\xi_{2}\right>\phi\operatorname{d}\!x.
  \end{equation*}
\end{enumerate}
\end{lemma}
Global $L^{p}$ estimate of \eqref{model} is stated in the following theorem.

\begin{lemma}\label{5-1}
Assume $A(x,z,\xi)$ satisfies \eqref{growth}. Let $F\in L^{p}(\Omega,\mathbb{R}^{n})$ and $u\in W_{0}^{1,p}(\Omega)$ is a weak solution of \eqref{model}, then
\begin{equation*}
  \int_{\Omega}|\nabla u|^{p}\operatorname{d}\!x\leq C \int_{\Omega}|F|^{p}\operatorname{d}\!x
\end{equation*}
Where $C=C(n,p,\Lambda)$
\end{lemma}
\begin{proof}
Let $u$ as a test function of \eqref{model}, we have
\begin{eqnarray*}
  \int_{\Omega}\left<A(x,u,\nabla u)-A(x,u,0),\nabla u\right>\operatorname{d}\!x
   &=&  \int_{\Omega}\left<A(x,u,\nabla u),\nabla u\right>\operatorname{d}\!x \\
   &=& \int_{\Omega}\left<|F|^{p-2}F,\nabla u\right>\operatorname{d}\!x \\
   &\leq& \int_{\Omega}|F|^{p-1}|\nabla u|\operatorname{d}\!x \\
   &\leq& \tau\int_{\Omega}|\nabla u|^{p}\operatorname{d}\!x +C(\tau)\int_{\Omega}|F|^{p}\operatorname{d}\!x
\end{eqnarray*}
for $\forall\tau>0$, where we used Young inequality. Applying Lemma\ref{5-2}, we get

\begin{eqnarray*}
  \int_{\Omega}|\nabla u|^{p}\operatorname{d}\!x
    &\leq& C^{*}\int_{\Omega}\left<A(x,u,\nabla u)-A(x,u,0),\nabla u\right>\operatorname{d}\!x \\
    &\leq& C^{*}\tau\int_{\Omega}|\nabla u|^{p}\operatorname{d}\!x +C(\tau)\int_{\Omega}|F|^{p}\operatorname{d}\!x
\end{eqnarray*}
Choose $\tau=\frac{1}{2C^{*}}$, we have
\begin{equation*}
  \int_{\Omega}|\nabla u|^{p}\operatorname{d}\!x\leq C \int_{\Omega}|F|^{p}\operatorname{d}\!x
\end{equation*}
\end{proof}

With these preliminary estimates at hand, we may now proceed to the proof of the weighted regularity estimate.
\begin{proof}[Proof of Theorem \ref{1-1}]
We will consider only the case $t\neq\infty$, as for $t=\infty$, the proof is similar. Let $N=N(n,p,\Lambda)$ be defined as in Corollary\ref{4-3}. For $q>1$, take $\epsilon_{1}=\epsilon\left(\frac{10}{1-4\delta}\right)^{nq}[\omega]_{q}^{2}$, $\epsilon_{2}=max\left\{1,2^{\frac{t}{pq}-1}\right\}\epsilon_{1}^{\frac{t}{pq}}$, choose $\epsilon$ sufficiently small such that
\begin{equation}\label{5.1}
  \epsilon_{2}\Gamma^{\frac{t}{p}}=\frac{1}{2}
\end{equation}
Let $\delta=\delta(n,p,\Lambda,\epsilon,\gamma)$ is determined by Corollary\ref{4-3}. Assume that the assumptions of Theorem\ref{1-1} hold with this choice of $\delta$. Furthermore, assume that $u$ is a weak solution of \eqref{model}, we select a finite collection of points $\{y_{i}\}_{i=1}^{L}\subset\overline{\Omega}$ and a ball $B$ such that $\overline{\Omega}\subset\cup_{i=1}^{L}B_{r}(y_{i})\subset B$, where $r=\frac{R}{400\sigma}$. We now prove Theorem\ref{1-1} with the following additional assumption that
\begin{equation}\label{5.2}
  \omega\left(\left\{x\in\Omega:\mathcal{M}(|\nabla u|^{p})>\left(\frac{6}{7}\right)^{n}\mu^{p}N^{p}\right\}\right)\leq \epsilon\omega(B_{r}(y_{i}))
\end{equation}
Where $\mu=\tilde{C}\|\nabla u\|_{L^{p}(\Omega)}$ with some sufficiently large constant $\tilde{C}$ depending on $n,p,q,\Lambda,\gamma,\Omega,\epsilon$ which is to be determined later. For $t\neq\infty$, we now consider the sum
\begin{equation}\label{5.3}
   S=\sum_{k=1}^\infty N^{tk}\omega\left(\left\{x\in \Omega:\mathcal{M}(|\nabla u|^{p})>\left(\frac{6}{7}\right)^{n}\mu^{p}N^{pk}\right\}\right)^{\frac{t}{pq}}
\end{equation}
Let $\Gamma=N^{p}>1$, then
\begin{equation*}
  S=\sum_{k=1}^\infty \Gamma^{\frac{tk}{p}}\omega\left(\left\{x\in \Omega:\mathcal{M}(|\nabla u|^{p})>\left(\frac{6}{7}\right)^{n}\mu^{p}\Gamma^{k}\right\}\right)^{\frac{t}{pq}}
\end{equation*}
Owing to \eqref{5.2} and applying Corollary\ref{4-3}, take $\beta=\frac{t}{pq}$ we have
\begin{eqnarray}\label{5.4}
   S&\leq&\sum_{k=1}^\infty \Gamma^{\frac{kt}{p}}\sum_{i=1}^k \epsilon_{2}^{i}\omega\left(\left\{x\in \Omega:\mathcal{M}(|F|^{p})>\left(\frac{6}{7}\right)^{n}\mu^{p}\delta^{p}\Gamma^{k-i}\right\}\right)^{\frac{t}{pq}}  \nonumber \\
   &+& \sum_{k=1}^\infty \Gamma^{\frac{tk}{p}}\epsilon_{2}^{k}\omega\left(\left\{x\in \Omega:\mathcal{M}(|\nabla u|^{p})>\left(\frac{6}{7}\right)^{n}\mu^{p}\right\}\right)^{\frac{t}{pq}}
\end{eqnarray}
To control $S$, we employ Fubini's theorem and Lemma\ref{2-1} to calculate:
\begin{eqnarray}
  S &\leq& \sum_{j=1}^\infty \left(\Gamma^{\frac{t}{p}}\epsilon_{2}\right)^{j}\sum_{k=j}^\infty \Gamma^{\frac{t(k-j)}{p}}\omega\left(\left\{x\in \Omega:\mathcal{M}(|F|^{p})>\left(\frac{6}{7}\right)^{n}\mu^{p}\delta^{p}\Gamma^{k-j}\right\}\right)^{\frac{t}{pq}} \nonumber\\
   &+& \sum_{k=1}^\infty \left(\Gamma^{\frac{t}{p}}\epsilon_{2}\right)^{k}\omega\left(\left\{x\in \Omega:\mathcal{M}(|\nabla u|^{p})>\left(\frac{6}{7}\right)^{n}\mu^{p}\right\}\right)^{\frac{t}{pq}} \nonumber\\
   &\leq& C \sum_{j=1}^\infty\left(\Gamma^{\frac{t}{p}}\epsilon_{2}\right)^{j}\left(\|\mathcal{M}(|F_{\mu}|^{p})\|_{L_{\omega}^{q,t/p}(\Omega)}^{t/p}+\omega(\Omega)^{\frac{t}{pq}}\right)
\end{eqnarray}
where $F_{\mu}=\frac{F}{\mu}$. Note that the choice of $\epsilon_{2}$, applying the Lemma\ref{2-1} again, we obtain
\begin{equation}\label{5.6}
  \|\mathcal{M}(|\nabla u_{\mu}|^{p})\|_{L_{\omega}^{q,t/p}(\Omega)}^{t/p}\leq C\left(\|\mathcal{M}(|F_{\mu}|^{p})\|_{L_{\omega}^{q,t/p}(\Omega)}^{t/p}+\omega(\Omega)^{t/pq}\right)
\end{equation}
for a constant $C$ depending on $n,p,\Lambda,t$, where $u_{\mu}=\frac{u}{\mu}$. Also, by the Lebesgue's differentiation theorem and the definition of weighted Lorentz space, we see that
\begin{eqnarray}\label{5.7}
   \|\nabla u\|_{L_{\omega}^{pq,t}(\Omega)}^{p}
   &=& \mu^{p}\||\nabla u_{\mu}|^{p}\|_{L_{\omega}^{q,t/p}(\Omega)} \nonumber\\
   &\leq& \mu^{p}\|\mathcal{M}(|\nabla u_{\mu}|^{p})\|_{L_{\omega}^{q,t/p}(\Omega)} \nonumber\\
   &\leq& C\mu^{p}\left(\|\mathcal{M}(|F_{\mu}|^{p})\|_{L_{\omega}^{q,t/p}(\Omega)}+\omega(\Omega)^{\frac{1}{q}}\right)
\end{eqnarray}
Using the last inequality and Lemma\ref{2-2}, we obtain
\begin{equation}\label{5.8}
  \|\nabla u\|^{p}_{L_{\omega}^{pq,t}(\Omega)}\leq C\mu^{p}\left(\||F_{\mu}|^{p}\|_{L_{\omega}^{q,t/p}(\Omega)}+\omega(\Omega)^{\frac{1}{q}}\right)
  =C\left(\|F\|^{p}_{L_{\omega}^{pq,t}(\Omega)}+\mu^{p}\omega(\Omega)^{\frac{1}{q}}\right)
\end{equation}
Owing to the definition of $\mu$ and Lemma\ref{5-1}, we get that
\begin{equation}\label{5.8.1}
  \mu^{p}\omega(\Omega)^{\frac{1}{q}}=\tilde{C}\omega(\Omega)^{\frac{1}{q}}\|\nabla u\|^{p}_{L^{p}(\Omega)}\leq\tilde{C}\omega(\Omega)^{\frac{1}{q}}\|F\|^{p}_{L^{p}(\Omega)}
\end{equation}
By appealing to Lemma\ref{jie}, we get that there exists a constant $s=s(n,q,\gamma)$ such that $q-s>1$ and $\omega\in A_{q-s}$ with $[\omega]_{q-s}\leq C(n,q,\gamma)$.
Hence, we can estimate $\|F\|^{p}_{L^{p}(\Omega)}$ as follows.
\begin{eqnarray*}
 \|F\|^{p}_{L^{p}(\Omega)}
   &=& \int_{\Omega}|F|^{p}\omega^{\frac{1}{q-s}}\omega^{-\frac{1}{q-s}}\operatorname{d}\!x   \\
   &\leq& \left(\int_{\Omega}\left(|F|^{p}\omega^{\frac{1}{q-s}}\right)^{q-s}\operatorname{d}\!x\right)^{\frac{1}{q-s}}
          \left(\int_{\Omega}\left(\omega^{-\frac{1}{q-s}}\right)^{\frac{q-s}{q-s-1}}\operatorname{d}\!x\right)^{\frac{q-s-1}{q-s}}  \\
   &=&  \left(\int_{\Omega}|F|^{p(q-s)}\omega\operatorname{d}\!x\right)^{\frac{1}{q-s}}
          \left(\int_{\Omega}\omega^{-\frac{1}{q-s-1}}\operatorname{d}\!x\right)^{\frac{q-s-1}{q-s}}   \\
   &=&  \|F\|^{p}_{L^{p(q-s)}_{\omega}(\Omega)}
         \left(\int_{\Omega}\omega^{-\frac{1}{q-s-1}}\operatorname{d}\!x\right)^{\frac{q-s-1}{q-s}}  \\
   &\leq& C\omega(\Omega)^{\frac{1}{q-s}-\frac{1}{q}}\|F\|^{p}_{L^{pq,\infty}_{\omega}(\Omega)}
         \left(\int_{\Omega}\omega^{-\frac{1}{q-s-1}}\operatorname{d}\!x\right)^{\frac{q-s-1}{q-s}}\\
   &\leq& C\omega(\Omega)^{-\frac{1}{q}}\|F\|^{p}_{L^{pq,t}_{\omega}(\Omega)}[\omega]_{q-s}^{\frac{1}{q-s}}  \\
   &\leq& C\omega(\Omega)^{-\frac{1}{q}}\|F\|^{p}_{L^{pq,t}_{\omega}(\Omega)}
\end{eqnarray*}
Where we used H\"older inequality and embedding theorem as mentioned in Lemma\ref{embedding}. Plugging this and \eqref{5.8.1} into \eqref{5.8}, we end up with
\begin{equation*}
  \|\nabla u\|_{L^{pq,t}_{\omega}(\Omega)}\leq C\|F\|_{L^{pq,t}_{\omega}(\Omega)}
\end{equation*}
Summarizing the efforts, we complete the proof of the Theorem as long as we can prove \eqref{5.2}.
Let
\begin{equation*}
  E:=\left\{x\in\Omega:\mathcal{M}(|\nabla u|^{p})>\left(\frac{6}{7}\right)^{n}\mu^{p}N^{p}\right\}
\end{equation*}
Owing to Lemma\ref{strong doubling}, we have the following estimates.
\begin{equation}\label{5.12}
  \frac{\omega(E)}{\omega(B_{r}(y_{i}))}
  =\frac{\omega(E)}{\omega(B)}\cdot\frac{\omega(B)}{\omega(B_{r}(y_{i}))}
  \leq\gamma\frac{\omega(E)}{\omega(B)}\left(\frac{|B|}{|B_{r}(y_{i})|}\right)^{q}
  \leq C(n,\gamma)\left(\frac{|E|}{|B|}\right)^{\alpha}\left(\frac{|B|}{|B_{r}(y_{i})|}\right)^{q}
\end{equation}
Where $\alpha$ is the constant as in Lemma\ref{strong doubling}. Then by weak (1,1)-type estimate for maximal functions, there exists a constant such that
\begin{equation}\label{5.13}
  |E|\leq\frac{C(n)}{(\mu N)^{p}}\int_{\Omega}|\nabla u|^{p}\operatorname{d}\!x=\frac{C(n,p,\Lambda)}{\widetilde{C}^{p}}
\end{equation}
It follows that
\begin{equation}\label{5.14}
   \frac{\omega(E)}{\omega(B_{r}(y_{i}))}\leq C(n,p,q,\Lambda,\gamma,\Omega,\epsilon)\widetilde{C}^{-p\alpha}
\end{equation}
Now, we choose $\widetilde{C}$ sufficiently large such that
\begin{equation*}
   \omega(E)\leq \epsilon\omega(B_{r}(y_{i}))
\end{equation*}
which gives estimate \eqref{5.2} as desired.
\end{proof}

\section{Besov regularity for solutions of a class of special harmonic equations.}\label{section5}

In this section, we study the Besov regularity for solutions of \eqref{specialmodel}, in the process, Calder\'on-Zygmund estimate will play an important role. For the sake of convenience and simplicity, we take advantage of Calder\'on-Zygmund estimate in a special case of $F=0$, $p=2$, $t=q$, $\omega=1$ and $\omega_{M}(t)=t^{\alpha}$. In this case, \eqref{c-condition} and \eqref{growth} can be rewritten as
\begin{equation}\label{3.1}
  \left<A(x,z,\xi)-A(x,z,\eta),\xi-\eta\right>\geq\Lambda^{-1}|\xi-\eta|^{2}
\end{equation}
\begin{equation}\label{3.2}
  |A(x,z,\xi)-A(x,z,\eta)|\leq\Lambda|\xi-\eta|
\end{equation}
and
\begin{equation}\label{3.3}
  |A(x,z_{1},\xi)-A(x,z_{2},\xi)|\leq|z_{1}-z_{2}|^{\alpha}|\xi|
\end{equation}

Given a domain $\Omega\subset\mathbb{R}^{n}$, we say that $f$ belongs to the local Besov space $B_{p,q,loc}^{\alpha}$ if $\varphi f$ belongs to the global Besov space $B_{p,q}^{\alpha}(\mathbb{R}^{n})$ for any $\varphi\in C_{0}^{\infty}(\Omega)$. Besides, we have the following technical lemma (cf.\cite{clop2019besov}).

\begin{lemma}
A function $f\in L_{loc}^{p}(\Omega)$ belongs to the local Besov space $B_{p,q,loc}^{\alpha}$ if and only if
\begin{equation*}
  \left\|\frac{\Delta_{h}f}{|h|^{\alpha}}\right\|_{L^{q}(\frac{\operatorname{d}\!h}{|h|^{n}})}<\infty
\end{equation*}
for any ball $B\subset2B\subset\Omega$ with radius $r_{B}$. Where $\Delta_{h}f(x)=f(x+h)-f(x)$. Here the measure $\frac{\operatorname{d}\!h}{|h|^{n}}$ is restricted to the ball $B(0,r_{B})$ on the $h$-space.
\end{lemma}

Next, we introduce some elementary estimates.
\begin{lemma}\label{3-3}
Suppose $1\leq p<\infty$, $u\in W^{1,p}(B_{R})$. Then, for each $0<\rho<R$, we have
\begin{equation*}
  \|\Delta_{h}u\|_{L^{p}(B_{\rho})}\leq C(n,p)|h|\|\nabla u\|_{L^{p}(B_{R})}
\end{equation*}
for all $0<|h|<\frac{R-\rho}{2}$.
\end{lemma}

\begin{lemma}
Let $A(x,z,\xi)$ satisfies \eqref{3.4}, \eqref{3.1}-\eqref{3.3}. Then $A(x,z,\xi)$ has small BMO semi-norm in $x$, i.e. \eqref{small BMO} holds.
\end{lemma}
\begin{proof}

\begin{eqnarray*}
  \fint_{B_{\rho}(y)}\theta(A,B_{\rho}(y))(x,z)\operatorname{d}\!x
   &=&\fint_{B_{\rho}(y)} \sup_{\xi\in\mathbb{R}^{n}\setminus\{0\}}\frac{|A(x,z,\xi)-\bar{A}_{B_{\rho}(y)}(z,\xi)|}{|\xi|}\operatorname{d}\!x \\
   &\leq& \fint_{B_{\rho}(y)} \sup_{\xi\in\mathbb{R}^{n}\setminus\{0\}}\fint_{B_{\rho}(y)}\frac{|A(x,z,\xi)-A(y,z,\xi)|}{|\xi|}\operatorname{d}\!y\operatorname{d}\!x \\
   &\leq&  \fint_{B_{\rho}(y)}\fint_{B_{\rho}(y)}(g(x)+g(y))|x-y|^{\alpha}\operatorname{d}\!y\operatorname{d}\!x \\
   &\leq&  \left(\fint_{B_{\rho}(y)}\fint_{B_{\rho}(y)}(g(x)+g(y))^{\frac{n}{\alpha}}\operatorname{d}\!y\operatorname{d}\!x\right)^{\frac{\alpha}{n}}
   \left(\fint_{B_{\rho}(y)}\fint_{B_{\rho}(y)}|x-y|^{\frac{n\alpha}{n-\alpha}}\operatorname{d}\!y\operatorname{d}\!x\right)^{\frac{n-\alpha}{n}} \\
   &\leq&  C(n,\alpha)\left(\int_{B_{\rho}(y)}g^{\frac{n}{\alpha}}\operatorname{d}\!x\right)^{\frac{\alpha}{n}}
\end{eqnarray*}
Where we used H\"older inequality. Thus, owing to the absolute continuity of the integral, we complete the proof.
\end{proof}

Now we proceed by proving Theorem \ref{1-2}
\begin{proof}[Proof of Theorem \ref{1-2}]
Fix a ball $B_{R}$ such that $B_{2R}\subset\subset\Omega$. Let $\eta\in C_{0}^{\infty}(B_{R})$ with $\eta=1$ on $B_{\frac{R}{2}}$ and $|\nabla\eta|\leq\frac{C}{R}$. For small enough $|h|$, given a test function $\varphi=\Delta_{-h}\left(\eta^{2}\triangle_{h}u\right)$, we test the equation\eqref{specialmodel} with $\varphi$, we have
\begin{equation*}
   \int_{\Omega}\left<A(x,u,\nabla u),\Delta_{-h}\nabla(\eta^{2}\Delta_{h}u)\right>\operatorname{d}\!x=0
\end{equation*}
Combine this and the ``integration-by-part'' formula for difference quotients, we get
\begin{equation}\label{3.5}
  \int_{\Omega}\left<\Delta_{h}A(x,u,\nabla u),\nabla(\eta^{2}\Delta_{h}u)\right>\operatorname{d}\!x=0
\end{equation}
We can write \eqref{3.5} as follows:
\begin{eqnarray*}
   &&\int_{\Omega}\left<A(x+h,u(x+h),\nabla u(x+h))-A(x+h,u(x+h),\nabla u(x)),\eta^{2}\nabla(\Delta_{h}u)\right>\operatorname{d}\!x      \nonumber\\
   &=&  -\int_{\Omega}\left<A(x+h,u(x+h),\nabla u(x+h))-A(x+h,u(x+h),\nabla u(x)),2\eta\nabla\eta\Delta_{h}u\right>\operatorname{d}\!x  \nonumber\\
   &+& \int_{\Omega}\left<A(x+h,u(x),\nabla u(x))-A(x+h,u(x+h),\nabla u(x)),\eta^{2}\nabla(\Delta_{h}u)\right>\operatorname{d}\!x \nonumber\\
   &+& \int_{\Omega}\left<A(x+h,u(x),\nabla u(x))-A(x+h,u(x+h),\nabla u(x)),2\eta\nabla\eta\Delta_{h}u\right>\operatorname{d}\!x \nonumber\\
   &+&  \int_{\Omega}\left<A(x,u(x),\nabla u(x))-A(x+h,u(x),\nabla u(x)),\eta^{2}\nabla(\Delta_{h}u)\right>\operatorname{d}\!x  \nonumber\\
   &+& \int_{\Omega}\left<A(x,u(x),\nabla u(x))-A(x+h,u(x),\nabla u(x)),2\eta\nabla\eta\Delta_{h}u\right>\operatorname{d}\!x  \nonumber\\
   &=& I_{1}+I_{2}+I_{3}+I_{4}+I_{5}
\end{eqnarray*}
Taking advantage of \eqref{3.1} in the left-hand side, we have
\begin{equation*}
  \Lambda^{-1}\int_{\Omega}|\Delta_{h}\nabla u|^{2}\eta^{2}\operatorname{d}\!x\leq |I_{1}|+|I_{2}|+|I_{3}|+|I_{4}|+|I_{5}|
\end{equation*}
Now, we estimate $I_{1}$-$I_{5}$ respectively. We proceed by estimating $I_{1}$ from \eqref{3.2} that
\begin{eqnarray*}
  |I_{1}| &\leq& 2\Lambda\int_{\Omega}|\Delta_{h}\nabla u||\eta||\nabla\eta||\Delta_{h}u|\operatorname{d}\!x \\
   &\leq& \epsilon\int_{\Omega}|\Delta_{h}\nabla u|^{2}\eta^{2}\operatorname{d}\!x +C(\epsilon,\Lambda)\int_{\Omega}|\nabla\eta|^{2}|\Delta_{h}u|^{2}\operatorname{d}\!x
\end{eqnarray*}
We use \eqref{3.3} and Young inequality as follows:
\begin{eqnarray*}
  |I_{2}| &\leq& \int_{\Omega}|\Delta_{h}u|^{\alpha}|\nabla u|\eta^{2}|\nabla(\Delta_{h}u)|\operatorname{d}\!x \\
   &\leq& \epsilon\int_{\Omega}|\Delta_{h}\nabla u|^{2}\eta^{2}\operatorname{d}\!x +C(\epsilon)\int_{\Omega}|\Delta_{h}u|^{2\alpha}|\nabla u|^{2}\eta^{2}\operatorname{d}\!x
\end{eqnarray*}
and
\begin{equation*}
  |I_{3}| \leq 2\int_{\Omega}|\Delta_{h}u|^{\alpha}|\nabla u|\eta|\nabla\eta||\Delta_{h}u|\operatorname{d}\!x
   = 2\int_{\Omega}|\Delta_{h}u|^{1+\alpha}|\nabla u|\eta|\nabla\eta|\operatorname{d}\!x
\end{equation*}
By virtue of assumption \eqref{3.4} and Young inequality, we have
\begin{eqnarray*}
  |I_{4}| &\leq& |h|^{\alpha}\int_{\Omega}(g(x+h)+g(x))|\nabla u(x)|\eta^{2}|\nabla(\Delta_{h}u)|\operatorname{d}\!x \\
   &\leq& \epsilon\int_{\Omega}|\nabla(\Delta_{h}u)|^{2}\eta^{2}\operatorname{d}\!x +C(\epsilon)|h|^{2\alpha}\int_{\Omega}(g(x+h)+g(x))^{2}|\nabla u(x)|^{2}\eta^{2}\operatorname{d}\!x
\end{eqnarray*}
and
\begin{eqnarray*}
  |I_{5}| &\leq& 2|h|^{\alpha}\int_{\Omega}(g(x+h)+g(x))|\nabla u(x)||\eta||\nabla\eta||\Delta_{h}u|\operatorname{d}\!x \\
   &\leq& C\int_{\Omega}|\Delta_{h}u|^{2}|\nabla\eta|^{2}\operatorname{d}\!x +C|h|^{2\alpha}\int_{\Omega}(g(x+h)+g(x))^{2}|\nabla u(x)|^{2}\eta^{2}\operatorname{d}\!x
\end{eqnarray*}
Collecting the above estimates, we get
\begin{eqnarray}\label{3.6}
  \int_{\Omega}|\Delta_{h}\nabla u|^{2}\eta^{2}\operatorname{d}\!x
   &\leq& C\int_{\Omega}|\nabla\eta|^{2}|\Delta_{h}u|^{2}\operatorname{d}\!x + C\int_{\Omega}|\Delta_{h}u|^{2\alpha}|\nabla u|^{2}\eta^{2}\operatorname{d}\!x \nonumber\\
   &+& C\int_{\Omega}|\Delta_{h}u|^{1+\alpha}|\nabla u|\eta|\nabla\eta|\operatorname{d}\!x  \nonumber\\
   &+&  C|h|^{2\alpha}\int_{\Omega}(g(x+h)+g(x))^{2}|\nabla u(x)|^{2}\eta^{2}\operatorname{d}\!x
\end{eqnarray}
From Lemma\ref{3-3} and the fact that $|\nabla\eta|\leq\frac{C}{R}$, the first term on the right-hand side can be estimated as:
\begin{equation*}
  \int_{B_{R}}|\nabla\eta|^{2}|\Delta_{h}u|^{2}\operatorname{d}\!x\leq\frac{|h|^{2}}{R^{2}}\int_{B_{R+|h|}}|\nabla u|^{2}\operatorname{d}\!x
\end{equation*}
Owing to H\"older inequality and Lemma\ref{3-3}, we obtain
\begin{eqnarray*}
  \int_{B_{R}}|\Delta_{h}u|^{2\alpha}|\nabla u|^{2}\eta^{2}\operatorname{d}\!x
   &\leq&  \left(\int_{B_{R}}|\Delta_{h}u|^{2}\operatorname{d}\!x\right)^{\alpha}\left(\int_{B_{R}}|\nabla u|^{\frac{2}{1-\alpha}}\operatorname{d}\!x\right)^{1-\alpha}\\
   &\leq&  C|h|^{2\alpha}\left(\int_{B_{R+|h|}}|\nabla u|^{2}\operatorname{d}\!x\right)^{\alpha}\left(\int_{B_{R}}|\nabla u|^{\frac{2}{1-\alpha}}\operatorname{d}\!x\right)^{1-\alpha}
\end{eqnarray*}
and
\begin{eqnarray*}
  \int_{B_{R}}|\Delta_{h}u|^{1+\alpha}|\nabla u|\eta|\nabla\eta|\operatorname{d}\!x
   &\leq& \left(\int_{B_{R}}|\Delta_{h}u|^{2}|\nabla\eta|^{\frac{2}{1+\alpha}}\operatorname{d}\!x\right)^{\frac{1+\alpha}{2}}\left(\int_{B_{R}}|\nabla u|^{\frac{2}{1-\alpha}}\operatorname{d}\!x\right)^{\frac{1-\alpha}{2}} \\
   &\leq& \frac{|h|^{1+\alpha}}{R}\left(\int_{B_{R+|h|}}|\nabla u|^{2}\operatorname{d}\!x\right)^{\frac{1+\alpha}{2}}\left(\int_{B_{R}}|\nabla u|^{\frac{2}{1-\alpha}}\operatorname{d}\!x\right)^{\frac{1-\alpha}{2}}
\end{eqnarray*}
The homogeneity of the equation together with Calder\'on-Zygmund estimate yield that $\nabla u\in L_{loc}^{s}(\Omega)$ for $\forall s>1$, see Theorem\ref{1-1-1} with $F=0$, $p=2$, $t=q$, $\omega=1$. In particular, $\nabla u\in L^{\frac{2}{1-\alpha}}(B_{R})$ and $\nabla u\in L^{\frac{2n}{n-2\alpha}}(B_{R})$. Thus, from H\"older inequality, we have
\begin{eqnarray*}
  \int_{B_{R}}(g(x+h)+g(x))^{2}|\nabla u(x)|^{2}\eta^{2}\operatorname{d}\!x
   &\leq& \left(\int_{B_{R}}(g(x+h)+g(x))^{\frac{n}{\alpha}}\operatorname{d}\!x \right)^{\frac{2\alpha}{n}}\left(\int_{B_{R}}|\nabla u|^{\frac{2n}{n-2\alpha}}\operatorname{d}\!x \right)^{\frac{n-2\alpha}{n}}\\
   &\leq&  C\left(\int_{B_{R+|h|}}g(x)^{\frac{n}{\alpha}}\operatorname{d}\!x \right)^{\frac{2\alpha}{n}}\left(\int_{B_{R}}|\nabla u|^{\frac{2n}{n-2\alpha}}\operatorname{d}\!x \right)^{\frac{n-2\alpha}{n}}
\end{eqnarray*}
Combining all this estimates and divide both side of \eqref{3.6} by $|h|^{2\alpha}$. Moreover, we use the fact that $\eta=1$ on $B_{\frac{R}{2}}$, then
\begin{eqnarray*}
  \int_{B_{\frac{R}{2}}}\left|\frac{\Delta_{h}\nabla u}{|h|^{\alpha}}\right|^{2}\operatorname{d}\!x
   &\leq& \frac{C|h|^{2-2\alpha}}{R^{2}}\int_{B_{R+|h|}}|\nabla u|^{2}\operatorname{d}\!x \\
   &+& C\left(\int_{B_{R+|h|}}|\nabla u|^{2}\operatorname{d}\!x\right)^{\alpha}\left(\int_{B_{R}}|\nabla u|^{\frac{2}{1-\alpha}}\operatorname{d}\!x\right)^{1-\alpha} \\
   &+& \frac{C|h|^{1-\alpha}}{R}\left(\int_{B_{R+|h|}}|\nabla u|^{2}\operatorname{d}\!x\right)^{\frac{1+\alpha}{2}}\left(\int_{B_{R}}|\nabla u|^{\frac{2}{1-\alpha}}\operatorname{d}\!x\right)^{\frac{1-\alpha}{2}} \\
   &+&  C\left(\int_{B_{R+|h|}}g(x)^{\frac{n}{\alpha}}\operatorname{d}\!x \right)^{\frac{2\alpha}{n}}\left(\int_{B_{R}}|\nabla u|^{\frac{2n}{n-2\alpha}}\operatorname{d}\!x \right)^{\frac{n-2\alpha}{n}}
\end{eqnarray*}
Now, we take supremum over all $h\in B_{\delta}$ for some $\delta<R$. Since $g\in L_{loc}^{\frac{n}{\alpha}}(\Omega)$, the proof of Theorem\ref{1-2} is complete.
\end{proof}


\end{document}